\documentclass{amsart}
\numberwithin{equation}{section}

\newtheorem{thm}{Theorem}[section]
\newtheorem{lem}[thm]{Lemma}

\newtheorem{defin}[thm]{Definition}

\begin{document}

\begin{center}
\textbf{{\large {\ Forward and Inverse Problems for a Langevin-Type Fractional Equation Involving  Non-Local Time Condition }}}\\
\medskip \textbf{Fayziev Yusuf $^{1},^{2},^{3}$, Jumaeva Shakhnoza $^{2}$}\\
\textit{fayziev.yusuf@mail.ru,     shahnozafarhodovna79@gmail.com}\\
\medskip \textit{\ $^{1}$ National University of Uzbekistan, Tashkent, Uzbekistan; \\ $^{2}$ V.~I.~Romanovskiy Institute of Mathematics, Uzbekistan Academy of Sciences,
\\$^{3}$Tashkent, Uzbekistan;  Karshi State University, Karshi, Uzbekistan }
\end{center}

\textbf{Abstract}: The paper addresses the formulation and analysis of direct and inverse problems for a  Langevin-type fractional differential equation under a non-local condition imposed on the time variable. An additional condition for solving the inverse problem has the
 form $u(t_0)=\omega$. The existence and uniqueness of a solution to the problems are established, and stability inequalities are derived. Criteria ensuring the uniqueness of the solution are identified.

\vskip 0.3cm \noindent {\it AMS 2000 Mathematics Subject
Classifications} :
Primary 35R11; Secondary 34A12.\\
{\it Key words}: Non-local time condition, Langevin-type equation, Caputo derivatives.

\section{Introduction}

The Langevin equation is a fundamental mathematical model in physics, formulated to describe systems influenced by random fluctuations. It is particularly effective in capturing stochastic dynamics in environments subject to thermal noise, such as those observed in Brownian motion. The classical Langevin equation, introduced by Paul Langevin in 1908 (see \cite{Langevin}), models Brownian motion by describing the velocity evolution of a particle subjected to friction and random forces. 

The study of fractional differential equations has become a significant area of research over the past few decades due to their ability to model memory and hereditary properties in various physical, biological, and engineering systems. Researchers have focused on developing theoretical aspects, solution methods, and applications for fractional differential equations.  The study of mathematical models has shifted from integer to fractional-order differential operators. For application details in physics, biology, chemistry, and medical sciences, please refer to the relevant works \cite{Magin}-\cite{Carvalho}. 

In recent years, there has been a growing interest in studying non-local, nonlinear, fractional-order boundary value problems, both single-valued and multivalued. The interest in this topic stems from the use of non-local and integral conditions in modelling real-world situations in applied and biological sciences. See [11-23] and related references for more information and examples.

Building on the widespread use of fractional calculus, the classical Langevin equation was modified by introducing a fractional-order derivative in place of the standard derivative, leading to its fractional form. Examples include motor single-file diffusion \cite{Eab}, the Kramers–Fokker–Planck equation and the Langevin equation \cite{Eule}, control systems \cite{West} and others. Lim et al. introduced a new form of the Langevin equation containing two different fractional orders, which has made it possible to investigate fractal processes in a more flexible environment in \cite{Lim}. For some recent works on the Langevin fractional equation, see the papers \cite{Aydin2023TwoOrder, Lmou2024Hybrid, Fa2023Nonlinear, Abdeljawad2023Hybrid, Zhao2022ML} and the references cited therein.

Let $H$ be a separable Hilbert space and $A: D(A) \rightarrow H$ be an arbitrary unbounded positive self-adjoint operator with the domain of definition $D(A)$. We assume that the operator $A$ has a complete orthonormal system of eigenfunctions $\{v_k\}$ and a countable set of positive eigenvalues $\lambda_k:0<\lambda_1 \leq \lambda_2...\rightarrow +\infty$. The sequence $\{\lambda_k\}$ has no finite limit points.

Let $C((a,b); H)$   be the set of continuous vector-valued functions $y(t)$ on $t\in (a,b)$ with values in $H$.

Let $AC[0, T]$ be the set of absolutely continuous functions defined on $[0, T]$ and let $AC([0, T]; H)$ stand for a space of absolutely continuous functions $y(t)$ with values in $H$ (see, \cite{Fomin} p.339).

In this paper, we consider the following non-local  boundary value problem for a Langevin equation with different order Caputo fractional derivatives: 
\begin{equation}\label{prob}
\begin{cases}
    D_{t}^{\beta}(D_t^{\alpha} u(t)) +D_t^{\beta} (A u(t))=f(t), \quad 0<t \leq T,
\\
    u(T)=\gamma u(0)+\varphi,
\\
D_t^{\alpha} u(+0)=\psi,
\end{cases}
\end{equation} where $0<\alpha<1$, $0<\beta<1$; $\varphi,\psi \in H$, $f(t) \in C([0,T];H)$, $\gamma$ is a constant, and $D_t^\alpha$,$D_t^\beta$ are Caputo fractional derivatives.

\section{Preliminaries}

In this section, we present several pieces of data about fractional derivatives, the Mittag-Leffler functions, and operators, which we will use below. 

The definitions of fractional integrals and derivatives for the function $h:\mathbb{R}_+ \rightarrow H$ are discussed in detail in \cite{Lizama}. The fractional integral of order $\sigma$ for a function $h(t)$ defined on $\mathbb{R}_+$ is given by:
$$
I_t^\sigma h(t)=\frac{1}{\Gamma(\sigma)} \int_0^t \frac{h(\xi)}{(t-\xi)^{1-\sigma}} d\xi, \quad t>0,
$$
where $\Gamma(\sigma)$ is the Euler gamma function. Using this definition, the Caputo fractional derivative of order $\gamma \in (0,1)$ can be defined as:
$$
D_t^\gamma h(t)=I_t^{1-\gamma} \frac{d}{dt}h(t).
$$
 Let $\epsilon$ be an arbitrary real number. The power of the operator 
$A$ is defined by the following:
$$
A^\epsilon h=\sum_{k=1}^\infty \lambda_k^\epsilon h_k v_k,
$$
where $h_k$ is the Fourier coefficient of the function $h \in H$, i.e., $h_k=(h,v_k)$. The domain of this operator is defined as:
$$
D(A^\epsilon)=\{h \in H:\sum_{k=1}^\infty \lambda_k^{2\epsilon} |h_k|^2<\infty\}.
$$
For elements of $D(A^{\epsilon})$ we introduce the norm
$$
||h||^2_\epsilon =\sum\limits_{k=1}^\infty \lambda_k^{2\epsilon} |h_k|^2.
$$
The function $$E_{\alpha,\mu} (z)=\sum_{n=0}^\infty \frac{z^n}{\Gamma(\alpha n+\mu)}$$  is called the Mittag-Leffler function with two parameters( see \cite{Dzha}, p 134), where  $0<\alpha<1$, $\mu \in \mathbb{C}$.

We present some asymptotic estimates for the Mittag-Leffler function:

\begin{lem}\label{Lemma1} 
Let $0<\alpha<1$ and $ \mu \in \mathbb{C}$. For any $t\geq 0$ one has (see \cite{Dzha}, p. 136)
$$
|E_{\alpha, \mu}(-t)|\leq \frac{C}{1+t},
$$
where  constant $C$  doesn't depend on $t$ and $\alpha$. 
\end{lem}

\begin{lem}\label{Lemma2}
Let $\alpha>0$ and $\lambda \in \mathbb{C}$ , then the following relation holds (see \cite{Kilbas}, p. 78):
$$D_t^\alpha E_{\alpha,1}(\lambda t^\alpha)=\lambda E_{\alpha,1}(\lambda t^\alpha).$$ 
\end{lem}

\begin{lem}\label{Lemma3}
For $Re \gamma >0$, $Re \mu >0$, $\lambda \in \mathbb{R}$, the following relation is valid (see \cite{Gor}, p. 87):
$$
    D_t^\gamma\left( t^{\mu-1}E_{\alpha,\mu}(\lambda t^\alpha)\right)=t^{\mu-\gamma-1} E_{\alpha,\mu-\gamma}(\lambda t^\alpha).
$$
\end{lem}

\begin{lem}\label{Lemma4} 
If $\alpha>0$, $\mu \in \mathbb{C}$, then the following recurrence relation holds (see \cite{Gor}, p. 57)
$$
    E_{\alpha,\mu}(z)=\frac{1}{\Gamma(\mu)}+zE_{\alpha,\mu+\alpha}(z).
$$
\end{lem}

\begin{lem}\label{Lemma5}
If $0<\alpha<1$ and $\mu$ be arbitrary complex number, one has (see \cite{Gor}, p. 61)
$$
    \int_0^t \eta^{\mu-1} E_{\alpha,\mu}(-\lambda \eta^\alpha) d\eta=t^\mu E_{\alpha,\mu+1}(-\lambda t^\alpha).
$$
\end{lem}

\begin{lem}\label{Lemma6}
If $0<\alpha<1$ and $\mu$ be arbitrary complex number, one has (see \cite{Dzha}, p. 134):
$$
E_{\alpha,\mu}(-z)=-\sum_{k=1}^n \frac{(-z)^{-k}}{\Gamma(\mu-k\alpha)}+O(z^{-n-1}).
$$
\end{lem}

\section{ Well-posedness of Problem \eqref{prob}}
First, let's begin by explaining the solution to the problem.

\begin{defin}\label{def1}
A function  $u(t)\in AC([0,T];H)$  is called a solution of problem  \eqref{prob} if $D_t^\beta (A u(t))$, $D_t^\beta(D_t^\alpha u(t)) \in C((0,T];H)$, $D_t^\alpha u(t) \in C([0,T];H)$ and $u(t)$ satisfies all conditions of problem \eqref{prob}.
\end{defin} 

\subsection{The auxiliary problems.}

To solve Problem \eqref{prob}, we divide it into two auxiliary problems:
\begin{equation}\label{aux1}
    \begin{cases}
        D_t^\beta(D^\alpha W(t))+D_t^\beta(AW(t))=0,\quad 0<t\leq T\\
        W(T)=\gamma W(0)+\Phi,    \\
        D_t^\alpha W(+0)=\psi
    \end{cases}
\end{equation}
and 
\begin{equation}\label{aux2}
    \begin{cases}
    D_t^\beta(D^\alpha V (t))+D_t^\beta(AV(t))=f(t),\quad 0<t\leq T\\
    V(0)=0,\\
    D_t^\alpha V(+0)=0,
    \end{cases}
\end{equation} where $\Phi \in H$ is a given function.

Problem \eqref{aux1} is a special case of Problem \eqref{prob}, and the solution of Problem \eqref{aux2} is defined similarly to Definition \ref{def1}.

\begin{lem} Let 
\begin{equation}
 \Phi=\varphi-V(T),  
\end{equation}  $W(t)$ and $V(t)$ be the corresponding solutions of Problems \eqref{aux1} and \eqref{aux2}. Then it is verified that the function $u(t)=W(t)+V(t)$ is a solution of problem \eqref{prob}.
\end{lem}

\begin{proof}
 Let  the functions $W(t)$ and $V(t)$ are solutions of the problems \eqref{aux1} and \eqref{aux2}, respectively, and let $\Phi=\varphi-V(T)$. Then

\begin{equation}\label{lauxeq}
    \begin{cases}
        D_t^\beta(D_t^\alpha W(t))+D_t^\beta(D_t^\alpha V(t))+D_t^\beta(AW(t))+D_t^\beta(AV (t))=f(t), \\
        W(T)=\gamma W(0)+\varphi-V(T), \\
        D_t^\alpha W(+0)+D_t^\alpha V(+0)=\psi.
    \end{cases}
\end{equation} 
Taking into account the equality $\gamma V(0)=0$, problem \eqref{lauxeq} has the following form:
$$
\begin{cases}
D_t^\beta\left(D_t^\alpha(W(t)+V(t))\right)+D_t^\beta\left(A(W(t)+V(t))\right)=f(t),\\
   W(T)+V(T)=\gamma(W(0)+V(0))+\varphi,\\
   D_t^\alpha(W(+0)+V(+0))=\psi.
\end{cases}
$$
If we note $ u = W + V$, then we obtain problem \eqref{prob}.
The lemma has been proved.
\end{proof} 

Applying the superposition technique, the original problem is decomposed into simpler subproblems. This widely used approach (see \cite{AshurovFayziev}, \cite{Sabitov}, etc.) reduces the task to solving auxiliary equations. Problem \eqref{aux2} is a non-homogeneous fractional differential equation subject to homogeneous boundary conditions.
For Problem \eqref{aux2}, we have the following statement.

\begin{thm}\label{theorem1}
Let $ f(t) \in C([0,T];D(A^\epsilon))$ for some $\epsilon \in(0,1) $. Then  the solution to Problem \eqref{aux2} is unique and defined in the following form:
\begin{equation} 
V(t)=\sum_{k=1}^\infty\bigg[\int_0^t (t-\eta)^{\alpha+\beta-1} E_{\alpha,\alpha+\beta}(-\lambda_k (t-\eta)^\alpha)f_k(\eta)  d\eta\bigg] v_k.
\end{equation}
In addition, there is a constant $C_\epsilon$ such that the following coercive-type inequality holds:
\begin{equation} 
||D_t^\beta(D_t^\alpha V(t))||^2+||D_t^\beta(V(t))||_1^2 \leq C_\epsilon \max_{t\in[0,T]}||f||_\epsilon^2,\quad 0<t\leq T.
\end{equation}
\end{thm}
This can be proved in the same way as the articles \cite{AshurovFayziev}, \cite{FayzievJumaeva}. 

\subsection{Solution of Problem \eqref{aux1}}

Due to the completeness of the system of eigenfunctions $\{v_k\}$ in $H$, the arbitrary solution can be written in the following form:
\begin{equation}\label{wt}
W(t)=\sum_{k=1}^\infty T_k(t) v_k,
\end{equation}
where  $T_k$ are the Fourier coefficients of the function $W(t)$ and the solutions of the non-local problem:
\begin{equation}\label{eqT}
    \begin{cases}
    D_t^\beta(D_t^\alpha T_k(t))+\lambda_k D_t^\beta T_k(t)=0, \quad 0<t\leq T,\\
    T_k(T)=\gamma T_k(+0)+\Phi_k, \\
    D_t^\alpha T_k(0)=\psi_k,
    \end{cases}
\end{equation} where $\Phi_k$ are the Fourier coefficients of the function $\Phi \in H$.

Let us denote $T_k(0)=b_k$. Then the solution of the differential equation \eqref{eqT} with initial conditions has the following form (see, e.g., \cite{FayzievJumaeva}):
\begin{equation}\label{tkt}
T_k(t)=b_k E_{\alpha,1}(-\lambda_k t^\alpha)+(\psi_k+\lambda_k b_k)t^\alpha E_{\alpha,\alpha+1}(-\lambda_k t^\alpha).
\end{equation}
According to the non-local condition of the differential equation \eqref{eqT}, we obtain the following equation:
\begin{equation}
    b_k E_{\alpha,1}(-\lambda_k T^\alpha)+(\psi_k+\lambda_k b_k)T^\alpha E_{\alpha,\alpha+1}(-\lambda_k T^\alpha)=\gamma b_k+\Phi.
\end{equation}
So, we obtain the equation that defines the unknown numbers $b_k$:
$$
    b_k \Bigg[E_{\alpha,1}(-\lambda_k T^\alpha)+\lambda_k T^\alpha E_{\alpha,\alpha+1}(-\lambda_k T^\alpha)-\gamma\Bigg]=\Phi_k-\psi_k T^\alpha E_{\alpha,\alpha+1}(-\lambda_k T^\alpha) .
$$
Now, by Lemma \ref{Lemma4}, we can rewrite the function $E_{\alpha,1}(-\lambda_k T^\alpha)$ as 
\begin{equation}\label{mittagrec}
E_{\alpha,1}(-\lambda_k T^\alpha)=\frac{1}{\Gamma(1)}-\lambda_k T^\alpha E_{\alpha,\alpha+1}(-\lambda_k T^\alpha).
\end{equation}
So, from this equality, we can rewrite $b_k$ in the following form:
\begin{equation}\label{bk1}
    b_k(1-\gamma)=\Phi_k-\psi_k T^\alpha E_{\alpha,\alpha+1}(-\lambda_k T^\alpha).
\end{equation}
Then, taking into account that $\gamma\neq 1$ we have the following formal solution for Problem \eqref{aux1}:
$$
W(t)=\sum_{k=1}^\infty\bigg [\frac{\Phi_k-\psi_k T^\alpha E_{\alpha,\alpha+1}(-\lambda_k T^\alpha)}{1-\gamma} E_{\alpha,1}(-\lambda_k t^\alpha)+$$$$ +\left(\psi_k+\lambda_k\frac{\Phi_k-\psi_k T^\alpha E_{\alpha,\alpha+1}(-\lambda_k T^\alpha)}{1-\gamma} \right)t^\alpha E_{\alpha,\alpha+1}(-\lambda_k t^\alpha)\bigg] v_k.
$$
By simplifying, we obtain
$$ W(t)=\sum_{k=1}^\infty\bigg[\frac{\Phi_k-\psi_k T^\alpha E_{\alpha,\alpha+1}(-\lambda_k T^\alpha)}{1-\gamma} \big(E_{\alpha,1}(-\lambda_k t^\alpha)+\lambda_k t^\alpha E_{\alpha,\alpha+1}(-\lambda_k t^\alpha)\big)+$$$$+\psi_k t^\alpha E_{\alpha,\alpha+1}(-\lambda_k t^\alpha)\bigg]v_k.$$ 

By using the equality \eqref{mittagrec}, we have the formal solution to Problem \eqref{aux1}, as
\begin{equation}\label{soluaux2}
    W(t)=\sum_{k=1}^\infty \bigg[\frac{\Phi_k-\psi_k T^\alpha E_{\alpha,\alpha+1}(-\lambda_k T^\alpha)}{1-\gamma}+\psi_k t^\alpha E_{\alpha,\alpha+1}(-\lambda_k t^\alpha)\bigg] v_k.
\end{equation}

Thus, we have obtained the result for Problem \eqref{aux1}.

\begin{thm}\label{theorem2}
Let $\psi\in H$ and $\Phi\in D(A)$. If $\gamma\neq 1$, then  Problem \eqref{aux1} has a unique solution, and this solution is defined in the form \eqref{soluaux2}. 
\noindent 
There is a constant $C$ such that the following coercive-type inequality holds:
$$
||D_t^\beta(D_t^\alpha W(t))||^2+||D_t^\beta W(t)||^2_1\leq Ct^{-2\beta}||\psi||^2+C||\Phi||_1^2, \quad 0<t\leq T.
$$
\end{thm}

\begin{proof}\textbf{Existence.}
We need to prove that function \eqref{soluaux2} satisfies all conditions of Definition 1 for the case $\gamma \neq 1$.  Let $S_j(t)$ be the partial sum of the series \eqref{soluaux2}:
$$
S_j(t)=\sum_{k=1}^j\bigg[\frac{\Phi_k-\psi_kT^\alpha E_{\alpha,\alpha+1}(-\lambda_k T^\alpha)}{1-\gamma}+\psi_kt^\alpha E_{\alpha,\alpha+1}(-\lambda_k t^\alpha)\bigg]v_k.
$$
First, we show that $AW(t)\in C((0,T],H)$. Then, by applying the operator $A$ on the partial sum $S_j(t)$, we have 
$$AS_j(t)=\sum_{k=1}^j \lambda_k \left[ \frac{\Phi_k}{1-\gamma}+\psi_k\left(t^\alpha E_{\alpha,\alpha+1}(-\lambda_k t^\alpha)-\frac{T^\alpha E_{\alpha,\alpha+1}(-\lambda_k T^\alpha)}{1-\gamma}\right)\right]v_k.$$
Applying Parseval's identity and the inequality $(a+b)^2\leq 2a^2+2b^2$, we get
$$
||AS_j(t)||^2\leq \sum_{k=1}^j\lambda_k^2 \left|\frac{\Phi_k}{1-\gamma}\right|^2+\sum_{k=1}^j\lambda_k^2 \left|\psi_k\left(t^\alpha E_{\alpha,\alpha+1}(-\lambda_k t^\alpha)-\frac{T^\alpha E_{\alpha,\alpha+1}(-\lambda_k T^\alpha)}{1-\gamma}\right)\right|^2$$$$=AS_j^1+AS_j^2,
$$ 
where
$$AS_j^1=\sum_{k=1}^j\lambda_k^2 \left|\frac{\Phi_k}{1-\gamma}\right|^2,$$$$AS_j^2=\sum_{k=1}^j\lambda_k^2 \left|\psi_k\left(t^\alpha E_{\alpha,\alpha+1}(-\lambda_k t^\alpha)-\frac{T^\alpha E_{\alpha,\alpha+1}(-\lambda_k T^\alpha)}{1-\gamma}\right)\right|^2. $$
If $\gamma \neq 1$, then the following estimate  holds true:
$$ AS_j^1\leq C\sum_{k=1}^j\lambda_k^2 |\Phi_k|^2.$$
According to Lemma \ref{Lemma2}, we obtain the estimate for $AS_j^2$ :$$AS_j^2\leq C\sum_{k=1}^j |\psi_k|^2.$$ Therefore, if $\Phi\in D(A)$ and $\psi \in H$, then $AW(t) \in C([0,T],H)$.

In the above reasoning, we can prove uniform convergence of the Fourier series corresponding to the function $W(t)$. So, we can conclude that if  $\Phi \in H$ and $\psi \in H$, then $W(t) \in AC([0,T]; H)$.

Now we verify $D_t^\beta(AW(t)) \in C((0,T];H)$. By applying Parseval's identity, we have 
$$
||D_t^\beta(AS_j(t))||^2= \sum_{k=1}^j\lambda_k^2\left|D_t^\beta\bigg[\frac{\Phi_k}{1-\gamma}-\psi_k\left(t^\alpha E_{\alpha,\alpha+1}(-\lambda_k t^\alpha)-\frac{T^\alpha E_{\alpha,\alpha+1}(-\lambda_k T^\alpha)}{1-\gamma}\right)\bigg]\right|^2.
$$
Since $\Phi_k$ does not depend on $t$, we apply the operator $D_t^\beta$ term by term to $AS_j(t)$. By using Lemma \ref{Lemma3}, we have the following equality:
$$ 
||D_t^\beta(AS_j(t))||^2=\sum_{k=1}^j \bigg|\psi_k \lambda_k t^{\alpha-\beta} E_{\alpha,\alpha-\beta+1}(-\lambda_k t^\alpha)\bigg|^2.
$$
Due to Lemma \ref{Lemma1}, we have the following estimate:
 $$
 ||D_t^\beta(AS_j(t))||^2\leq C \sum_{k=1}^j t^{-2\beta}|\psi_k|^2.
 $$
Thus, if $\psi \in H$, then $D_t^\beta(AW(t)) \in C((0,T),H)$ is valid.
The equation of Problem \eqref{aux1} can be written in the form $D_t^\beta(D_t^\alpha W(t))=-D_t^\beta(AW(t))$. Therefore, from the above reasoning, we have $D_t^\beta(D_t^\alpha W(t))\in C((0,T];H)$.

For proving $D_t^\alpha(W(t))\in C([0,T],H)$, we apply the operator $D_t^\alpha$ to the formal series $S_j(t)$. Due to Parseval's identity, we obtain the following: 
$$||D_t^\alpha(S_j(t))||^2=\sum_{k=1}^j \bigg|D_t^\alpha\bigg[\frac{\Phi_k}{1-\gamma}-\psi_k\left(t^\alpha E_{\alpha,\alpha+1}(-\lambda_k t^\alpha)-\frac{T^\alpha E_{\alpha,\alpha+1}(-\lambda_k T^\alpha)}{1-\gamma}\right)\bigg]\bigg|^2.$$ 
Using the same way that we used for estimating $D_t^\beta$, Lemmas \ref{Lemma1} and \ref{Lemma3}, we have:
$$
||D_t^\alpha(S_j(t)||^2=\sum_{k=1}^j \bigg|\psi_k E_\alpha(-\lambda_k t^\alpha)\bigg|^2\leq C\sum_{k=1}^j|\psi_k|^2.
$$
Hence, if $\psi \in H$, then $D_t^\alpha W(t) \in C([0,T];H)$.

It is easy to verify that the derivative of the solution to Problem \eqref{aux1} is a continuous function. Therefore, this solution is an absolutely continuous function. So, according to Definition \ref{def1}, $W(t)$ is a solution to  Problem \eqref{aux1}.

\textbf{Uniqueness.} We use the standard method for proving the uniqueness of this solution. Assume that there are two different solutions to  Problem \eqref{aux1}: $W_1(t)$ and $W_2(t)$. Then, the function $W(t)=W_1(t)-W_2(t)$ will be the solution to the following problem:
\begin{equation}
    \begin{cases}
        D_t^\beta(D_t^\alpha W(t))+D_t^\beta(AW(t))=0, \quad 0<t\leq T \\
        W(T)=\gamma W(0), \\
        D_t^\alpha W(0)=0.
    \end{cases}
\end{equation}
Let $W_k(t)=(W(t),v_k)$. According to the completeness of the set of  eigenfunctions $\{v_k\}$ in $H$,we have the following non-local problem for $W_k(t)$:
\begin{equation}\label{uniqeq}
    \begin{cases}
        D_t^\beta(D_t^\alpha W_k(t))+\lambda_k D_t^\beta(W_k(t))=0, \\
        W_k(T)=\gamma W_k(0), \\
        D_t^\alpha W_k(0)=0.
    \end{cases}
\end{equation}
Let's denote $W_k(0)=b_k$. Then the unique solution to the problem \eqref{uniqeq} has the form (see \cite{FayzievJumaeva})
$$
W_k(t)=b_k E_{\alpha,1}(-\lambda_k t^\alpha)+\lambda_k b_k t^\alpha E_{\alpha,\alpha+1}(-\lambda_k t^\alpha).
$$
From the non-local condition of the problem \eqref{uniqeq} and the recurrence relation for the Mittag-Leffler function, we obtain the following equation to find the unknown $b_k$:
\begin{equation}\label{bk}
    b_k=\gamma b_k.
\end{equation}
If $\gamma \neq 1$, then all $b_k$ are equal to zero and so $W_k(t)=0$ for all $k, k\in \mathbb{N}$. As stated by the completeness of the set of eigenfunctions $\{v_k\}$, we can deduce that $W(t) \equiv 0$. Thus,  Problem \eqref{aux1} in this case has a unique solution.
\end{proof}

Now suppose that $\gamma=1$. Then, any arbitrary value $b_k$ satisfies the equation \eqref{bk}. So, for this case, there is no uniqueness of the solution to Problem \eqref{aux1}.

\subsection{Solution to main problem}

As mentioned above, if we put $\Phi=\varphi-V(T) \in H$ and $W(t)$ and $V(t)$ in the corresponding solutions to Problems \eqref{aux1} and \eqref{aux2}, then $u(t)=W(t)+V(t)$ is a solution to problem \eqref{prob}. Thus, when $\gamma\neq1$ the solution of the problem \eqref{prob}  has the following form
\begin{equation}\label{solu}
u(t)=\sum_{k=1}^\infty\bigg[\frac{\varphi_k-V_k(T)-\psi_k T^\alpha E_{\alpha,\alpha+1}(-\lambda_k T^\alpha)}{1-\gamma}+\psi_k t^\alpha E_{\alpha,\alpha+1}(-\lambda_k t^\alpha)+ V_k(t)\bigg] v_k
\end{equation} where
$$
V_k(t)=\int_0^t  (t-\eta)^{\alpha+\beta-1} E_{\alpha,\alpha+\beta}(-\lambda_k (t-\eta)^\alpha)f_k(\eta)  d\eta.
$$
The uniqueness of the solution $u(t)$ follows from the uniqueness of the solutions $W(t)$ and $V(t)$.
Finally, we have had the result for the Main Problem \eqref{prob}.

\begin{thm}\label{Theorem3}
Let $\varphi \in D(A)$, $\psi \in H$ and $f(t)\in C([0,T];D(A^{\epsilon+1})$ for some $\epsilon \in (0,1)$. If $\gamma\neq1$, then problem \eqref{prob} has a unique solution, and this solution has the form \eqref{solu}.
\noindent
Moreover, there are constants $ C_\epsilon, C $ such that  the following coercive type inequality holds:
$$
\|D_t^{\beta}(D_t^\alpha u(t))\|^2 + \|D_t^\beta u(t)\|_{1}^{2} \leq t^{-2\rho} \left( C_\epsilon \max_{t \in [0,T]} \|f\|_{\epsilon+1}^{2} + C \|\psi\|^{2} \right).
$$
\end{thm}

This theorem is proven based on the theorems obtained for problems \eqref{aux1} and \eqref{aux2}. If $f$ does not depend on $t$, the assertion of Theorem \ref{Theorem3} remains valid for all $f \in H$.

\textbf{Corollary 1.} \textit{Let $\varphi \in D(A)$, $\psi \in H$ and $f\in H$. If $\gamma\neq 1$, then Problem \eqref{prob} has a unique solution, and this solution has the representation 
$$
u(t)=\sum_{k=1}^\infty\bigg[\frac{\varphi_k-f_k T^{\alpha+\beta}E_{\alpha,\alpha+\beta+1}(-\lambda_k T^\alpha)-\psi_kT^\alpha E_{\alpha,\alpha+1}(-\lambda_k T^\alpha)}{1-\gamma}$$ 
\begin{equation}\label{corol1}
+\psi_k t^\alpha E_{\alpha,\alpha+1}(-\lambda_k t^\alpha) +f_k t^{\alpha+\beta}E_{\alpha,\alpha+\beta+1}(-\lambda_k t^\alpha)\bigg]v_k.
\end{equation}}

\begin{proof} Since $f$ does not depend on $t$, then by using Lemma \ref{Lemma5}, we have the following form for the Fourier coefficients of $V(t)$:
$$V_k(t)=f_k t^{\alpha+\beta} E_{\alpha,\alpha+\beta+1}(-\lambda_k t^\alpha).$$ 
Let $S_j(t)$ be the partial sum of the series \eqref{corol1}. Then, by Lemma \ref{Lemma1}, we obtain 
$$
||AS_j||^2\leq C_1\sum_{k=1}^j \lambda_k^2 |\varphi_k|^2 +C_2\sum_{k=1}^j |\psi_k|^2+C_3\sum_{k=1}^j |f_k|^2.
$$
By making use of this estimate and also following reasoning analogous to that in the proof of Theorem \ref{Theorem3}, it can be easily verified that \eqref{corol1} solves Problem \eqref{prob} under conditions of Corollary 1.
\end{proof}

\section{Inverse problem}

We now turn our attention to the inverse problems. In this context, we assume that the unknown function $f \in H$ is independent of $t$. To obtain a solution, we impose the following additional condition:
\begin{equation}\label{conad}
    u(t_0)=\omega, \quad 0<t_0<T,
\end{equation} where $\omega \in H$ is a given element.

\begin{defin}\label{def2}
A pair $\{u(t), f\}$ of functions, where $u(t) \in AC([0,T]; H)$ and $f \in H$, such that $D_t^\beta (A u(t))$, $D_t^\beta(D_t^\alpha u(t)) \in C((0,T]; H)$ and $D_t^\alpha u(t) \in C([0,T]; H)$, and which satisfies the conditions of Problem \eqref{prob} and \eqref{conad}, is called a solution of the inverse problem.
\end{defin}

Based on  \eqref{solu}  and Lemma \ref{Lemma5}, the solution to the inverse problem, in the case $\gamma \neq 1$, can be represented in the following form:
$$
    u(t)=\sum_{k=1}^\infty \bigg[\frac{\varphi_k-f_k T^{\alpha+\beta}E_{\alpha,\alpha+\beta+1}(-\lambda_k T^\alpha)-\psi_k T^\alpha E_{\alpha,\alpha+1}(-\lambda_k T^\alpha)}{1-\gamma}
$$
\begin{equation}
+\psi_k t^\alpha E_{\alpha,\alpha+1}(-\lambda_k t^\alpha)+f_k t^{\alpha+\beta}E_{\alpha,\alpha+\beta+1}(-\lambda_k t^\alpha)\bigg]v_k.
\end{equation}
Therefore, due to the additional condition \eqref{conad}  and the completeness of the system $\{v_k\}$, we obtain
$$
\frac{\varphi_k}{1-\gamma}+\psi_k\bigg(t_0^\alpha E_{\alpha,\alpha+1}(-\lambda_k t_o^\alpha)-\frac{T^\alpha E_{\alpha,\alpha+1}(-\lambda_k T^\alpha)}{1-\gamma}\bigg)$$
$$+f_k \bigg(t_0^{\alpha+\beta} E_{\alpha,\alpha+\beta+1}(-\lambda_k t_0^\alpha)-\frac{T^{\alpha+\beta}E_{\alpha,\alpha+\beta+1}(-\lambda_k T^\alpha)}{1-\gamma}\bigg)=\omega_k,
$$
where $\omega_k$ are the Fourier coefficients of the given element $\omega$.

A simple calculation yields:
$$
f_k \bigg[(1-\gamma)t_0^{\alpha+\beta}E_{\alpha,\alpha+\beta+1}(-\lambda_k t_0^\alpha)-T^{\alpha+\beta}E_{\alpha,\alpha+\beta+1}(-\lambda_k T^\alpha)\bigg]$$
\begin{equation}\label{fksol}
=(1-\gamma)\omega_k-\varphi_k +\psi_k \bigg(T^\alpha E_{\alpha,\alpha+1}(-\lambda_k T^\alpha)-(1-\gamma)t_0^\alpha E_{\alpha,\alpha+1}(-\lambda_k t_0^\alpha\bigg) .
\end{equation}
Hereafter, we will use the notation
$$
\Delta_k(\gamma,t_0,T)=(1-\gamma)t_0^{\alpha+\beta}E_{\alpha,\alpha+\beta+1}(-\lambda_k t_0^\alpha)-T^{\alpha+\beta}E_{\alpha,\alpha+\beta+1}(-\lambda_k T^\alpha), \quad k\geq 1.$$

If $\Delta_k(\gamma,t_0, T)\neq 0$ for all $k$, then the unknown coefficients $f_k$ are found uniquely, and the inverse problem also has a unique solution. 

If $\Delta_k(\gamma,t_0, T)=0$ for some $k_0$, then the inverse problem will not have a unique solution. Now, for set $\Omega\subset \mathbb{R}^3$, we denote 
$$K_0:=K_0(\Omega)=\{k:\Delta_k(\gamma,t_0, T)= 0,\quad k\in\mathbb{N}, \quad (\gamma,t_0,T)\in\Omega\}.$$

\begin{lem}\label{Lemma4.1} 
If $t_0<T$, then there exists a number $P\neq0$ such that for large $k$ the following equality holds:
$$
\frac{t_0^{\alpha+\beta} E_{\alpha,\alpha+\beta+1}(-\lambda_k t_0^\alpha)}{T^{\alpha+\beta}E_{\alpha,\alpha+\beta+1}(-\lambda_k T^\alpha)}=\bigg(\frac{t_0}{T}\bigg)^\beta\Bigg[1+\frac{P}{\lambda_k}+O\left(\frac{1}{\lambda_k^2}\right)\Bigg].
$$
\end{lem}
\begin{proof} By Lemma \ref{Lemma6} (with $n=2$) for large $k$, we have 
$$
\frac{t_0^{\alpha+\beta} E_{\alpha,\alpha+\beta+1}(-\lambda_k t_0^\alpha)}{T^{\alpha+\beta}E_{\alpha,\alpha+\beta+1}(-\lambda_k T^\alpha)}=\frac{t_0^{\alpha+\beta}\bigg[\frac{(-\lambda_k t_0^\alpha)^{-1}}{\Gamma(\beta+1)}+\frac{(-\lambda_k t_0^\alpha)^{-2}}{\Gamma(\beta-\alpha+1)}+O((\lambda_k t_0^\alpha)^{-3}\bigg]}{T^{\alpha+\beta}\bigg[\frac{(-\lambda_k T^\alpha)^{-1}}{\Gamma(\beta+1)}+\frac{(-\lambda_k T^\alpha)^{-2}}{\Gamma(\beta-\alpha+1)}+O((\lambda_k T^\alpha)^{-3}\bigg]}
$$
$$
=\frac{t_0^\beta\bigg[\frac{1}{\Gamma(\beta+1)}+\frac{(-\lambda_k t_0^\alpha)^{-1}}{\Gamma(\beta-\alpha+1)}+O((\lambda_k t_0^\alpha)^{-2})\bigg]}{T^\beta\bigg[\frac{1}{\Gamma(\beta+1)}+\frac{(-\lambda_k T^\alpha)^{-1}}{\Gamma(\beta-\alpha+1)}+O((\lambda_k T^\alpha)^{-2})\bigg]}
$$
$$
=\bigg(\frac{t_o}{T}\bigg)^\beta\Bigg[1+\frac{\frac{(-\lambda_k t_0^\alpha)^{-1}}{\Gamma(\beta-\alpha+1)}-\frac{(-\lambda_k T^\alpha)^{-1}}{\Gamma(\beta-\alpha+1)}+O((\lambda_k t_0^\alpha)^{-2})-O((\lambda_k T^\alpha)^{-2})}{\frac{1}{\Gamma(\beta+1)}+\frac{(-\lambda_k T^\alpha)^{-1}}{\Gamma(\beta-\alpha+1)}+O((\lambda_k T^\alpha)^{-2})}\Bigg]
$$
$$
=\bigg(\frac{t_0}{T}\bigg)^\beta \Bigg[1+\frac{\frac{C_1}{\lambda_k}(t_0^{-\alpha}-T^{-\alpha})+O(\lambda_k^{-2})}{C_2+\frac{C_1}{\lambda_k T^\alpha}+O(\lambda_k^{-2})}\Bigg]
$$
$$
=\bigg(\frac{t_0}{T}\bigg)^\beta \Bigg[1+\frac{\frac{C_1}{\lambda_k}(t_0^{-\alpha}-T^{-\alpha})}{C_2+\frac{C_1}{\lambda_k T^-{\alpha}}+O(\lambda_k^{-2})}+\frac{O(\lambda_k^{-2})}{C_2+\frac{C_1}{\lambda_k T^\alpha}+O(\lambda_k^{-2})}\Bigg]
$$
$$
=\bigg(\frac{t_0}{T}\bigg)^\beta \Bigg[1+\frac{\frac{C_1}{C_2\lambda_k}(t_0^{-\alpha}-T^{-\alpha})\bigg(1+\frac{C_1}{C_2 \lambda_k T^\alpha}+O(\lambda_k^{-2})-\frac{C_1}{C_2 \lambda_k T^\alpha}-O(\lambda_k^{-2})\bigg)}{1+\frac{C_1}{C_2\lambda_k T^{-\alpha}}+O(\lambda_k^{-2})}+\frac{O(\lambda_k^{-2})}{C_2+\frac{C_1}{\lambda_k T^\alpha}+O(\lambda_k^{-2})}\Bigg]
$$
$$
=\bigg(\frac{t_0}{T}\bigg)^\beta \Bigg[1+\frac{C_1(t_0^{-\alpha}-T^{-\alpha})}{C_2}\frac{1}{\lambda_k}+O(\lambda_k^{-2})+\frac{O(\lambda_k^{-2})}{C_2+\frac{C_1}{\lambda_k T^\alpha}+O(\lambda_k^{-2})}\Bigg]
$$
$$
=\bigg(\frac{t_0}{T}\bigg)^\beta\Bigg[1+\frac{P}{\lambda_k}+O(\lambda_k^{-2})\Bigg],
$$
where 
$$
C_1=\frac{1}{\Gamma(\beta-\alpha+1)}, \quad  C_2=\frac{1}{\Gamma(\beta+1)}, \quad P=\frac{\Gamma(\beta-\alpha+1)}{\Gamma(\beta+1)}(t_0^{-\alpha}-T^{-\alpha}).
$$
\end{proof}

\begin{lem}\label{Lemma4.2}
The following  cases are possible:
\\
a) If $\gamma>1$, then  $K_0$ is empty and the following estimate holds:
\begin{equation}\label{deltak}
    |\Delta_k(\gamma,t_0,T)|\geq\frac{C}{\lambda_k}, \quad k\geq 1;
\end{equation} 
b) Let $\gamma<1$. If  condition $T^\beta< t_0^\beta|\gamma-1|$ holds, then the set $K_0$
is empty, and the estimate \eqref{deltak} is valid. Otherwise, that is, if condition $T^\beta\geq t_0^\beta |\gamma-1|$ holds, then the set $K_0$ is non-empty, but if it  can only have a finite number of elements, then  there exists a number $k_0$ such that for all $k\geq k_0$ the following estimate holds  
\begin{equation}\label{lambdak2}
|\Delta_k(\gamma,t_0,T)|\geq\frac{C}{\lambda_k^2}.
\end{equation}
\end{lem}

This lemma can be proved using the same method as in the article \cite{AshurovNuralieva}.

Using estimates of the denominator $\Delta_k(\gamma,t_0, T)$, we can investigate the existence and uniqueness of the solution to the inverse problem.

\begin{thm}\label{theorem4.1}
Let $\omega$, $\psi$, $\varphi \in D(A)$. If either $\gamma>1$ holds, or both $\gamma<1$ and $T^\beta< t_0^\beta|\gamma-1|$ are satisfied concurrently. Then, the inverse problem has a unique solution $\{u(t),f\}$ given by:
$$
    u(t)=\sum_{k=1}^\infty \bigg[\frac{\varphi_k-f_k T^{\alpha+\beta}E_{\alpha,\alpha+\beta+1}(-\lambda_k T^\alpha)-\psi_k T^\alpha E_{\alpha,\alpha+1}(-\lambda_k T^\alpha)}{1-\gamma}
$$
\begin{equation}\label{solinv}
+\psi_k t^\alpha E_{\alpha,\alpha+1}(-\lambda_k t^\alpha)+f_k t^{\alpha+\beta}E_{\alpha,\alpha+\beta+1}(-\lambda_k t^\alpha)\bigg]v_k
\end{equation}
 and
\begin{equation}\label{finvser}
    f=\sum_{k=1}^\infty f_k v_k,
\end{equation} where
$$
    f_k=\frac{(1-\gamma)\omega_k-\varphi_k +\psi_k \bigg(T^\alpha E_{\alpha,\alpha+1}(-\lambda_k T^\alpha)-(1-\gamma)t_0^\alpha E_{\alpha,\alpha+1}(-\lambda_k t_0^\alpha\bigg)}{\Delta_k(\gamma,t_0,T)} ,
$$
  $\varphi_k$, $\psi_k$ and $\omega_k$ are the Fourier coefficients of the elements  $\varphi$, $\psi$ and $\omega$,respectively.
\end{thm}

\begin{proof}\textbf{ Existence.} Let's show the convergence of series \eqref{finvser}. If $F_j$ are the partial sums of series \eqref{finvser}, then according to the Parseval identity, we may write the following equality:
$$
||F_j||^2=\sum_{k=1}^j\Bigg|\frac{(1-\gamma)\omega_k-\varphi_k +\psi_k \bigg(T^\alpha E_{\alpha,\alpha+1}(-\lambda_k T^\alpha)-(1-\gamma)t_0^\alpha E_{\alpha,\alpha+1}(-\lambda_k t_0^\alpha\bigg)}{\Delta_k(\gamma,t_0,T)}\Bigg|^2
$$
$$
\leq 3\sum_{k=1}^j \Bigg|\frac{(1-\gamma)\omega_k}{\Delta_k(\gamma,t_0,T)}\Bigg|^2
+3\sum_{k=1}^j\Bigg|\frac{\varphi_k}{(\Delta_k(\gamma,t_0,T)}\Bigg|^2
$$$$+3\sum_{k=1}^j\Bigg|\frac{\psi_k (T^\alpha E_{\alpha,\alpha+1}(-\lambda_k T^\alpha)-(1-\gamma)t_0^\alpha E_{\alpha,\alpha+1}(-\lambda_k t_0^\alpha)}{\Delta_k(\gamma,t_0,T)}\Bigg|^2
=3F_{j,1}+3F_{j,2}+3F_{j,3} ,$$
where
$$
F_{j,1}=\sum_{k=1}^j \Bigg|\frac{(1-\gamma)\omega_k}{\Delta_k(\gamma,t_0,T)}\Bigg|^2,
\quad
F_{j,2}=\sum_{k=1}^j\Bigg|\frac{\varphi_k}{\Delta_k(\gamma,t_0,T)}\Bigg|^2,
$$
$$
F_{j,3}=\sum_{k=1}^j\Bigg|\frac{\psi_k \bigg(T^\alpha E_{\alpha,\alpha+1}(-\lambda_k T^\alpha)-(1-\gamma)t_0^\alpha E_{\alpha,\alpha+1}(-\lambda_k t_0^\alpha\bigg)}{\Delta_k(\gamma,t_0,T)}\Bigg|^2.
$$
From  Lemma \ref{Lemma1} and Case b of Lemma \ref{Lemma4.2}, we gave the following estimates for $F_{j, i},i=1,2,3$:
$$F_{j,1}\leq C \sum_{k=1}^j \lambda_k^2 |\omega_k|^2, \quad F_{j,2}\leq C \sum_{k=1}^j \lambda_k^2|\varphi_k|^2,$$
$$F_{j,3}\leq C \sum_{k=1}^j \lambda_k^2 |\psi_k|^2.$$
Thus, if $\omega$, $\varphi$, $\psi\in D(A)$, then from the estimates of $F_{j,i}$, we obtain $f\in H$.

After finding the unknown function $f \in H$, the fulfilment of the conditions of Definition \ref{def2} for function $u(t)$, defined by series \eqref{solinv}, is proved in the same way as with Theorem \ref{theorem2}.

\textbf{Uniqueness.}
To prove the uniqueness of the solution, we use the standard technique, that is, the solution of problem \eqref{prob} with the homogeneous condition (i.e, $\varphi_k=0$, $\psi_k=0$ and $\omega_k=0$)  is identically zero. Then, due to \eqref{finvser}, it follows that $f_k\equiv0$ for all $k\geq1$. This follows from the completeness of the system $\{v_k\}$, in which case we obtain $ f\equiv 0$ and $u(t)\equiv 0$, as required.

The theorem has been proved.
\end{proof}

Eventually, when $\Delta_k(\gamma,t_0, T)=0$ for some $k \in \mathbb{N}$ or $K_0$ is non-empty, we will only require the existence of a solution $\{u(t),f\}$.

\begin{thm}\label{theorem4.2}

Let $\omega$, $\psi$, $\varphi \in D(A^2)$,  $\gamma<1$ and $T^\beta\geq t_0^\beta |\gamma-1|$, for some $\gamma, t_0$ and $T$.
Then, for the existence of the solution $\{u(t),f\}$ to the inverse problem, it is necessary and sufficient that the following condition 
\begin{equation}
    (1-\gamma)\omega_k-\varphi_k +\psi_k \bigg(T^\alpha E_{\alpha,\alpha+1}(-\lambda_k T^\alpha)-(1-\gamma)t_0^\alpha E_{\alpha,\alpha+1}(-\lambda_k t_0^\alpha\bigg)=0
\end{equation}  is satisfied.
In this case  a pair of solutions $\{u(t),f\}$ is not unique and has the following representations:
$$
    u(t)=\sum_{k\notin K_0} \bigg[\frac{\varphi_k-f_k T^{\alpha+\beta}E_{\alpha,\alpha+\beta+1}(-\lambda_k T^\alpha)-\psi_k T^\alpha E_{\alpha,\alpha+1}(-\lambda_k T^\alpha)}{1-\gamma}
$$
$$
+\psi_k t^\alpha E_{\alpha,\alpha+1}(-\lambda_k t^\alpha)+f_k t^{\alpha+\beta}E_{\alpha,\alpha+\beta+1}(-\lambda_k t^\alpha)\bigg]v_k
$$
$$
+\sum_{k\in K_0}\bigg[\frac{\varphi_k-\bar f_k T^{\alpha+\beta}E_{\alpha,\alpha+\beta+1}(-\lambda_k T^\alpha)-\psi_k T^\alpha E_{\alpha,\alpha+1}(-\lambda_k T^\alpha)}{1-\gamma}
$$
\begin{equation}
+\psi_k t^\alpha E_{\alpha,\alpha+1}(-\lambda_k t^\alpha)+\bar f_k t^{\alpha+\beta}E_{\alpha,\alpha+\beta+1}(-\lambda_k t^\alpha)\bigg]v_k,
\end{equation}
\begin{equation}\label{fknonempty}
f=\sum_{k\in K_0} \bar f_k v_k+\sum_{k\notin K_0}f_k v_k,
\end{equation}
where $$f_k=\frac{(1-\gamma)\omega_k-\varphi_k -\psi_k\Delta_k(\gamma,t_0,T)}{\Delta_k(\gamma,t_0,T)},$$
arbitrary coefficients $\bar f_k,k\in K_0$.
\end{thm}

\begin{proof}
To establish the theorem, it is necessary to verify that the series \eqref{fknonempty} fulfills all the requirements specified in Definition \ref{def2}. This result follows directly from the proof of Theorem \ref{theorem4.1}, and the argument remains nearly identical regardless of which of the conditions is satisfied.
The series \eqref{fknonempty} is composed of two parts. The first part, as established by Lemma \ref{Lemma4.2}, is a finite sum of smooth functions.
For the second part, we can verify that the series satisfies the conditions of Definition \ref{def2} using the same approach as for series \eqref{finvser}. In this case, the lower bound \eqref{lambdak2} for $\Delta_k(\gamma,t_0,T)$ is applied.
\end{proof}


\begin{thebibliography}{99}
\bibitem{Langevin} 
Langevin, P., Sur la théorie du mouvement brownien, \textit{C. R. Acad. Sci. Paris}, \textbf{146} (1908), 530--533.

\bibitem{Magin} 
Magin, R. L., \textit{Fractional Calculus in Bioengineering}, Begell House, Chicago, 2006.

\bibitem{Mainardi} 
Mainardi, F., \textit{Fractional Calculus and Waves in Linear Viscoelasticity}, Imperial College Press, Singapore, 2010.

\bibitem{Fallahgoul} 
Fallahgoul, H. A., Focardi, S. M., Fabozzi, F. J., \textit{Fractional Calculus and Fractional Processes with Applications to Financial Economics}, Elsevier, London, 2017.

\bibitem{Carvalho} 
Carvalho, A., Pinto, C. M. A., A delay fractional order model for the co-infection of malaria and HIV/AIDS, \textit{Int. J. Dyn. Control}, \textbf{5} (2017), 168--186.

\bibitem{Ahmadivasundaram} 
Ahmad, B., Sivasundaram, S., On four-point non-local boundary value problems of nonlinear integro-differential equations of fractional order, \textit{Appl. Math. Comput.}, \textbf{217}(2) (2010), 480--487.

\bibitem{Fomin} 
Kolmogorov, A., Fomin, S., \textit{Introductory Real Analysis}, Courier Corporation, 1975.

\bibitem{Lizama} 
Lizama, C., Abstract linear fractional evolution equations, in: \textit{Handbook of Fractional Calculus with Applications}, vol. 2 (2019), 465--498.

\bibitem{Dzha} 
Dzherbashian, M., \textit{Integral Transforms and Representation of Functions in the Complex Domain}, Nauka, Moscow, 1966 (in Russian).

\bibitem{Gao} 
Gao, Z., Yu, X., Wang, J., Non-local problems for Langevin-type differential equations with two fractional-order derivatives, \textit{Bound. Value Probl.}, \textbf{2016} (2016), Article ID 52.

\bibitem{Eule} 
Eule, S., Friedrich, R., Jenko, F., Kleinhans, D., Langevin's approach to fractional diffusion equations including inertial effects, \textit{J. Phys. Chem. B}, \textbf{111}(39) (2007), 11474--11477.

\bibitem{Eab} 
Eab, C. H., Lim, S. C., Fractional generalized Langevin equation approach to single-file diffusion, \textit{Physica A}, \textbf{389}(13) (2010), 2510--2521.

\bibitem{West} 
West, B. J., Latka, M., Fractional Langevin model of gait variability, \textit{J. Neuroeng. Rehabil.}, \textbf{2} (2005), 1--9.

\bibitem{Lim} 
Lim, S. C., Li, M., Teo, L. P., Langevin equation with two fractional orders, \textit{Phys. Lett. A}, \textbf{372}(42) (2008), 6309--6320.

\bibitem{Aydin2023TwoOrder} 
Aydin, M., Mahmudov, N. I., On the exact solution of two-order time-fractional Langevin differential equations with delay, \textit{Math. Methods Appl. Sci.}, \textbf{46}(2) (2023), 904--918.

\bibitem{Lmou2024Hybrid} 
Lmou, A., Hilal, Y., Kajouni, A., On new fractional Langevin equations involving hybrid $\psi$-Caputo derivatives with non-local boundary conditions, \textit{Bound. Value Probl.}, \textbf{2024} (2024), Article ID 75.

\bibitem{Fa2023Nonlinear} 
Fa, K. S., Pianegonda, S., Nonlinear Langevin equation with separable coefficients: ergodicity, generalized Einstein relation, and fluctuation theorems, \textit{arXiv preprint} (2023), arXiv:2311.17069.

\bibitem{Abdeljawad2023Hybrid} 
Abdeljawad, T., Panchal, S., A hybrid fractional Langevin system with three-point boundary conditions and $\psi$-Hilfer derivative, \textit{Bound. Value Probl.}, \textbf{2023} (2023), Article ID 111.

\bibitem{Zhao2022ML} 
Zhao, Y., On a Langevin equation with Mittag-Leffler fractional derivative and delay effects, \textit{Mathematics}, \textbf{10}(12) (2022), Article ID 725.

\bibitem{AhmadNieto} 
Ahmad, B., Nieto, J. J., Alsaedi, A., El-Shahed, M., A study of nonlinear Langevin equation involving two fractional orders in different intervals, \textit{Nonlinear Anal. Real World Appl.}, \textbf{13}(2) (2012), 599--606.

\bibitem{AhmadNtouyas} 
Ahmad, B., Ntouyas, S. K., New existence results for differential inclusions involving Langevin equation with two indices, \textit{J. Nonlinear Convex Anal.}, \textbf{14}(3) (2013), 437--450.

\bibitem{AshurovFayziev} 
Ashurov, R., Fayziev, Yu., On the Non-Local Problems in Time for Time-Fractional Subdiffusion Equations, \textit{Fractal Fract.}, \textbf{6}(1) (2022), 41--62.

\bibitem{Kilbas} 
Kilbas, A. A., Srivastava, H. M., Trujillo, J. J., \textit{Theory and Applications of Fractional Differential Equations}, Elsevier, Amsterdam, 2006.

\bibitem{Gor} 
Gorenflo, R., Kilbas, A. A., Mainardi, F., Rogozin, S., \textit{Mittag-Leffler Functions, Related Topics and Applications}, Springer, Berlin-Heidelberg, 2014.

\bibitem{FayzievJumaeva} 
Fayziev, Yu., Jumaeva, Sh., On the Cauchy problem for the Langevin-type fractional equation, \textit{Uzbek Math. J.}, \textbf{69}(2) (2015), 109--118.

\bibitem{Sabitov} 
Sabitov, K. B., \textit{Equations of Mathematical Physics}, Fizmatlit, Moscow, 2013 (in Russian).

\bibitem{AshurovNuralieva} 
Ashurov, R. R., Nuraliyeva, N. S., A Three-Parameter Problem for Fractional Differential Equation with an Abstract Operator, \textit{Lobachevskii J. Math.}, \textbf{45} (2024), 5788--5801.

\end{thebibliography}
\end{document}